\documentclass[12pt]{article}

\usepackage{amsmath,amsfonts,amssymb,amsthm,graphicx,float, hyphenat}

\def\d#1{{#1\kern-0.4em\char"16\kern-0.1em}}
\def\D#1{{\raise0.2ex\hbox{-}\kern-0.4em #1}}

\theoremstyle{plain}

\newtheorem{definition}{Definition}[section]
\newtheorem{proposition}[definition]{Proposition}
\newtheorem{theorem}[definition]{Theorem}

\newtheorem{lemma}[definition]{Lemma}
\newtheorem{corollary}[definition]{Corollary}

\newtheorem{example}[definition]{Example}
\newtheorem{remark}[definition]{Remark}

\usepackage{xcolor}
\definecolor{Magenta}{rgb}{1,0,1}

\newcommand{\id}{{\sf e }}
\newcommand{\QA}{{\mathcal Q}_{\A_0}^{\,\C}(\A)}

\newcommand{\ip}[2]{\langle{#1}|{#2}\rangle}
\newcommand{\mult}{\,{\scriptstyle \square}\,}
\newcommand{\ad}{^{\mbox{\scriptsize $\dag$}}}
\newcommand{\X}{{\mathcal X}}
\newcommand{\DD}{\mathfrak D}
\newcommand{\LXH}{{\mathcal L}\ad(\DD,\X)}
\newcommand{\Lc}{{\mathcal L}}
\newcommand{\LD}{{\mathcal L}\ad(\D)}
\newcommand{\up}{\raisebox{0.7mm}{$\upharpoonright$}}

\def\N{{\mathcal{N}}}
\def\R{{\mathcal{R}}}
\def\B{{\mathcal{B}}}
\def\C{{\mathcal{C}}}

\def\D{{\cal{D}}}

\def\A{{\mathcal{A}}}
\def\S{{\mathcal{S}}}

\def\H{{\mathcal{H}}}

\def\M{\mathfrak{M}}
\def\CC{\mathbb{C}}
\def\RR{\mathbb{R}}
\def\ZZ{\mathbb{Z}}
\def\NN{\mathbb{N}}

\author{Stefan Ivkovi\'c, Bogdan D. Djordjevi\'c and Giorgia Bellomonte}
\date{}

\begin{document}

\title{On representations and topological aspects of positive maps on non-unital quasi *- algebras}

\maketitle

\begin{abstract}
In this paper we provide a representation of a certain class of C*-valued positive sesquilinear and linear maps on non-unital quasi *\hyp algebras, thus extending the results from \cite{GBSICT} to the case of non-unital quasi *-algebras. Also, we illustrate our results on the concrete examples of non-unital Banach quasi *-algebras, such as the standard Hilbert module over a commutative C*-algebra, Schatten p-ideals and noncommutative $L^2$-spaces induced by a semifinite, nonfinite trace. As a consequence of our results, we obtain a representation of all bounded positive linear  C*-valued  maps on non-unital C*-algebras. We also deduce some norm inequalities for these maps. Finally, we consider a noncommutative $L^2$-space equipped with the topology generated by a positive sesquilinear form and we construct a topologically transitive operator on this space.
\end{abstract}


{\bf Keyphrases:} *-representations, positive C*-valued maps, non-unital quasi *-algebras,   noncommutative $L^2$-spaces,  topologically transitive operators.


\section{Introduction}Positive maps on operator algebras play an important role in various areas of mathematics and physics. Many deep mathematical results have been obtained in this topic, see for example \cite{CDPR} and \cite{stormer}. These maps turn out to be a powerful tool for characterizing entanglement in quantum information theory (for more details and the literature, see  \cite{CDPR}). Moreover, as observed in the introduction of \cite{stormer}, another area where positive maps appear is the one of operator spaces. Although, in that context, the maps are usually completely positive, there is a close connection with positive maps on C*-algebras. These facts provide a motivation for obtaining representations (thus in a certain sense an explicit description) of such maps. The classical result in representation theory of positive maps is the famous Stinespring theorem which gives a representation of completely positive maps on unital C*-algebras and, quite recently, in \cite[Corollary 3.10]{GBSICT}, a representation of general positive C*-valued maps on arbitrary unital *-algebras has also appeared. \\

In this paper (in Corollary \ref{cor: repres omega}) we obtain a representation of general bounded, positive linear C*- valued maps on non-unital C*-algebras.As a consequence, in Corollary \ref{cor: 3.10} we deduce also certain norm inequalities for these maps. Moreover, in Corollary \ref{moreGNS alg} we provide representations of positive C*-valued maps on general non-unital normed *-algebras with a right approximate identity.\\

 It is now long time that the C*-algebraic approach to quantum theories has been considered as a too rigid framework where casting all objects of physical interest. Therefore, several possible generalizations have been proposed: quasi *-algebras, partial *-algebras and so on. They reveal in fact to be more suited to cover situations where unbounded operator algebras are involved. Further, quasi *-algebras are also interesting from the mathematical point of view because several beautiful and important mathematical structures (such as commutative and noncommutative $L^p$ spaces) are examples of quasi *-algebras.  For all these reasons, it might be also relevant to study positive C*-valued maps on quasi *-algebras.  In \cite[Corollary 3.13]{GBSICT}, a representation of a special class of bounded positive linear C*-valued maps on unital normed quasi *-algebras has been provided. This class of positive maps turns out to be a proper generalization of the class of  bounded positive linear  C*-valued maps on C*-algebras, as we argue in Remark \ref{rem: 3.12} below.

However, there are many important examples of non-unital normed quasi *-algebras, such as the standard Hilbert module over a commutative C*-algebra, the Schatten p-ideals, and noncommutative $L^2$-space induced by a semifinite non-finite trace. These examples have motivated us to extend the theory from \cite{GBSICT} to the case of non-unital normed quasi *-algebras. In Corollary \ref{cor:quasibanach} we provide an extension of  \cite[Corollary 3.13]{GBSICT} to the  the case of  non-unital normed quasi *-algebras with a right approximate identity. To sum up, the main purpose and aim of Section 3 is to provide representations of bounded  positive C*-valued maps both on 
 non-unital operator algebras and on non-unital Banach quasi *-algebras. The main difference compared to the unital case is the lack of the existence of the cyclic vector. Therefore, the representations of positive C*-valued maps on unital (quasi) *-algebras obtained in \cite{GBSICT} are no longer valid in the non-unital case because those representations  are expressed in terms of the cyclic vector (see \cite[Corollary 3.10]{GBSICT} and \cite[Corollary 3.13]{GBSICT}) and cyclic vector is induced by the unit in the respective (quasi) *-algebra. For this reason, the non-unital case requires a to a certain extent different approach. Although we do apply the general GNS-construction from \cite{GBSICT}, some additional algebraic work and modifications are needed in the construction of *-representations for positive C*-valued maps on non-unital (quasi) *-algebras and this is exactly what the proofs in Section 3 provide.  In Section 4 we give examples of positive   C*-valued maps on concrete non-unital normed quasi *-algebras. \\

Every positive linear map on a *-algebra $\A$ induces in a natural way the corresponding positive sesquilinear map on $\A \times  \A$. On quasi *-algebras, because of the lack of an everywhere defined multiplication, in representation theory  it is sometimes more convenient to study positive sesquilinear maps instead of positive linear maps, as observed e.g. in \cite{MFCT}.
   
   In this paper, we provide first representations of  bounded positive  sesquilinear C*-valued maps on non-unital normed quasi *-algebras and then, in various corollaries, we obtain the above-mentioned representations of  bounded positive linear C*-valued maps.   Now, in addition for being a convenient tool in representation theory,  bounded  positive sesquilinear maps generate also several topologies on normed quasi *-algebras, as observed in \cite[Remark 3.1.32]{MFCT}. On the other hand, the topological dynamics of linear operators is in general an active research field in mathematics, see \cite{A,C,cd18}, and topological transitivity is one of the main concepts in this theory. The topological transitivity of cosine operator functions acting on a solid Banach function space (thus in particular on commutative $L^2$-spaces) has been studied  for instance in \cite{tsi}, whereas the topological dynamics of multipliers on Schatten p-ideals has been studied in \cite{IT2}, \cite{YRH} and \cite{ZLD}. Motivated by these facts, in Section \ref{sec: Inf VN algebras} we consider a noncommutative $L^2$-space equipped with a topology generated by a bounded positive sesquilinear form and we construct a topologically transitive right multiplier operator on this space. Also, we prove that the cosine operator function generated by this operator is topologically transitive as well.\\

The paper is organized as follows. In Section \ref{sec: Basic defn}, we recall some notions, concepts and results mainly from \cite{GBSICT} and \cite{MFCT} that are needed for the rest of the paper. In Section \ref{sec: construction repres}, we provide (general) representations of  a certain class of  bounded positive sesquilinear or linear C*-valued maps on non-unital normed quasi *-algebras possessing a certain {\em right approximate identity} (see Definition \ref{defn: rai}). In Section \ref{sec: appl and ex} we illustrate our results by concrete examples. For this purpose, in Example \ref{ex: 4.11} we construct positive linear map from a non-unital noncommutative $L^2$-space into the space of continuous functions on a closed interval, and it turns out that this operator can be considered as a (proper) generalization of an integral operator on a commutative $L^2$-space, as we argue in Remark \ref{genint}. In fact, most of our examples in Section 4 are devoted to the study of integral operators and their generalizations in the context of various normed non-unital quasi *-algebras. In Section \ref{sec: Inf VN algebras}, we consider topologically transitive operators on non commutative $L^2$-spaces equipped with a topology generated by a positive sesquilinear form. 

\section{Basic definitions and results}\label{sec: Basic defn}

A {\em quasi *-algebra} $(\A, \A_0)$ is a pair consisting of a vector space $\A$ and a *-algebra $\A_0$ contained in $\A$ as a subspace and such that
\begin{itemize}
	\item $\A$ carries an involution $a\mapsto a^*$ extending the involution of $\A_0$;
	\item $\A$ is  a bimodule over $\A_0$ and the module multiplications extend the multiplication of $\A_0$. In particular, the following associative laws hold:
	\begin{equation}\notag \label{eq_associativity}
		(xa)y = x(ay); \ \ a(xy)= (ax)y, \quad \forall \ a \in \A, \  x,y \in \A_0;
	\end{equation}
	\item $(ax)^*=x^*a^*$, for every $a \in \A$ and $x \in \A_0$.
\end{itemize}


We will always suppose that
\begin{align*}
	&ax=0, \; \forall x\in \A_0 \Rightarrow a=0 \\
	&ax=0, \; \forall a\in \A \Rightarrow x=0. 
\end{align*}
Clearly, both these conditions are automatically satisfied if $(\A, \A_0)$ has an identity $\id$.

\begin{definition}
    A  quasi *-algebra $(\A, \A_0)$ is said to be  {\em locally convex} if $\A$ is a locally convex vector space, with a topology $\tau$ enjoying the following properties
\begin{itemize}
	\item[{\sf (lc1)}] $c\mapsto c^*$, \ $c\in\A_0$,  is continuous;
	\item[{\sf (lc2)}] for every $a \in \A$, the maps  $c \mapsto ac$ and
	$c \mapsto ca$, from $\A_0$ into $\A$, $c\in \A_0$, are continuous;
	\item[{\sf (lc3)}] $\overline{\A_0}^\tau = \A$; i.e., $\A_0$ is dense in $\A[\tau]$.
\end{itemize}
The involution of $\A_0$ extends by continuity to $\A$. 
Moreover, if $\tau$ is a norm topology, with norm $\|\cdot\|$, and
\begin{itemize}
	\item[{\sf (bq*)}] $\|a^*\|=\|a\|, \; \forall a \in \A$
\end{itemize}
then, $(\A, \A_0)$ is called a {\em normed  quasi *-algebra} and a {\em Banach  quasi *-algebra} if the normed vector space $\A[\|\cdot\|]$ is complete.
\end{definition}

Let $\C$ be a $C^*-$algebra with norm $\|\cdot\|_\C$ and positive cone $\C^+$.\\

Let $\A$ be a complex vector space  
and $\S$ a  positive  sesquilinear $\C$-valued map on   $\A\times\A$  $$\S:(a,b)\in\A\times\A\to\S(a,b)\in\C;$$ i.e., a map with the properties  \begin{itemize}
	\item[$i)$] $\S(a,a)\in\C^+$,
	\item[$ii)$]$\S(\alpha a+\beta b,\gamma c)=\overline{\gamma}[\alpha\S( a,c)+\beta \S(b,c)]$,  
\end{itemize}
with $a,b,c \in\A$ and $\alpha,\beta,\gamma\in\mathbb{C}$. \\
By property $i)$ it follows that \begin{itemize}
	\item[$iii)$]  $\S(b,a)=\S(a,b)^*$, for all $a,b\in\A$.
\end{itemize}

\begin{lemma}\cite[Lemma 2.3]{GBSICT}\label{lemma 1} Let $\S$ be a positive  sesquilinear $\C$-valued map $\S$ on   $\A\times\A$. Then,  $$\|\S(a,b)\|_\C\leq2\|\S(a,a)\|_\C^{1/2}\|\S(b,b)\|_\C^{1/2}, \quad\forall a,b\in\A.$$   If $\C$ is commutative, then
		 $\S$ satisfies the Cauchy-Schwarz inequality:$$\|\S(a,b)\|_\C\leq\|\S(a,a)\|_\C^{1/2}\|\S(b,b)\|_\C^{1/2},\quad \forall a,b\in\A.$$
	\end{lemma}
\bigskip
A positive  sesquilinear $\C$-valued map $\S$ is called {\em faithful} if $$\S(a,a)=0 \;\Rightarrow\; a=0.$$ 

\begin{definition}
    \label{defn: C*-valued quasi inner product}
Let $\A$ be a complex vector space. A  faithful positive  sesquilinear $\C$-valued map $\S$ on   $\A\times\A$ is said to be a {\em C*-valued quasi inner product} and we often will write $\ip{a}{b}_\S:=\S(a,b)$,  $a,b\in\A$.
\end{definition}
A C*-valued quasi inner product $\S:\A\times\A\to \C$    induces a quasi norm $\|\cdot\|_\S$ on $\A$:  

\label{vp norm}  \begin{equation*}
	\label{eq: norm phi def}
	\|a\|_\S:=\sqrt{\|\ip{a}{a}_\S\|_\C}=\sqrt{\|\S(a,a)\|_\C},\quad a\in\A.\end{equation*} 

\medskip
This means that
\begin{align} 	
	& \| a\|_\S\geq0, \quad \forall a\in\A \mbox{ and }\| a\|_\S=0\Leftrightarrow a=0,\nonumber\\
	& \|\alpha a\|_\S=|\alpha|\| a\|_\S, \qquad \forall \alpha\in\mathbb{C}, a\in\A,\nonumber\\ 
	&   \|a+b\|_\S\leq\sqrt{2}(\| a\|_\S+\| b\|_\S), \qquad \forall a,b\in\A.\label{eqn:Q3}
\end{align} The  inequality \eqref{eqn:Q3}  is due to Lemma \ref{lemma 1} (for details see \cite{GBSICT}).


The space $\A$ is then a quasi normed space w.r.to the quasi norm $\|\cdot\|_\S$.

\begin{definition} \cite{GBSICT}
	If the vector space  $\A$ is complete w.r.to the quasi norm $\|\cdot\|_\S$, it will be called a {\em quasi Banach space with $\C$-valued quasi inner product} or for short a {\em quasi $B_\C$-space}.
\end{definition}

\begin{remark}
   If $\S$ satisfies the Cauchy-Schwarz inequality (e.g. if either  $\C$ is commutative or both $\A$ is a bimodule over $\C$ and $\S$ is $\C$-linear), then we get a normed space and not just a quasi normed space, see also \cite{GBSICT}.
\end{remark}
Let $\X$ be a quasi $B_\C$-space and $\DD(X)$ a dense subspace of $\X$.
A linear map $X:\DD(X)\to \X$ is {\em $\S$-adjointable} if there exists a linear map $X^*$ defined on a subspace $\DD(X^*)\subset \X$ such that
$$\S(Xa,b)= \S(a, X^*b), \quad \forall a \in \DD(X), b\in \DD(X^*)$$
and $X^*$ is said  the $\S$-adjoint of $X$.\\

Let $\DD$ be a dense subspace of $\X$ and let us consider the following families of linear operators acting on $\DD$ \cite{GBSICT}:
\begin{align*}		{\LXH}&=\{X \mbox{ $\S$-adjointable}, \DD(X)=\DD;\; \DD(X^*)\supset \DD\} \\	{\Lc^\dagger(\DD)}&=\{X\in \LXH: X\DD\subset \DD; \; X^*\DD\subset \DD\} \\	{\Lc^\dagger(\DD)_b} &=\{Y\in \Lc^\dagger(\DD); \, {Y} \mbox{ is bounded on $\DD$} \}.\end{align*}
The involution in $\LXH$ is defined by		$X^\dag := X^*\upharpoonright \D$, the restriction of the $\S$-adjoint $X^*$ to $\DD$.

The sets $\LD$ and ${\Lc^\dagger(\DD)_b}$ are *-algebras.

\begin{remark} If $X\in \LXH$ then $X$ is closable (see \cite[Remark 2.8]{GBSICT}). 

\end{remark}

\begin{remark} \label{rem_partialmult}
	$\LXH$ is also a {\em partial *-algebra} \cite{ait_book}  with respect to the following operations: the usual sum $X_1 + X_2 $,
	the scalar multiplication $\lambda X$, the involution $ X \mapsto X\ad := X^* \up {\DD}$ and the \emph{(weak)}
	partial multiplication 
	defined whenever there   exists $Y\in \LXH$ such that
	$$\S(X_2 a,X_1b)= \S (Ya,b), \quad \forall a,b \in \DD.$$
	The element $Y$, if it exists, is unique and  $Y=X_1\mult X_2$.
\end{remark}

\bigskip
If $\S$ is not faithful, we can consider the set \begin{equation}\label{eq: NS}
    N_\S=\{a\in\A:\, \S(a,a)=0\}.
\end{equation}
The following two lemmas are proved in \cite{GBSICT}.
\begin{lemma}\label{N phi subspace}
	$ \N_\S$ is a subspace of $\A$.
\end{lemma} 

We define 
a positive  sesquilinear $\C$-valued map on   $\A/\N_\S\times\A/\N_\S$ as follows: \begin{equation*}\ip{\cdot}{\cdot}_\S:\A/\N_\S\times\,\A/\N_\S\to\C\end{equation*} \begin{equation}\label{eq: inner pr phi semidef}
\ip{a+\N_\S}{ b+\N_\S}_\S:=\S(a,b).
\end{equation}The associated quasi norm is:\begin{equation}
\label{eq: norm phi semidef}
\|a+\N_\S\|_\S:=\sqrt{\|\S(a,a)\|_\C},\quad a\in\A.\end{equation}

It is easy to check that
\begin{lemma} $\A/\N_\S[\N_\S]$ is a quasi normed space.\end{lemma}

\begin{definition} \label{px}  Let $(\A,\A_0)$ be a  quasi *-algebra. We denote by $\QA$ the set of all  positive sesquilinear $\C$-valued maps on $\A \times \A$ that satisfy a property of invariance:
	\begin{itemize}
		\item[(I)]
		$\S(ax,y)=\S(x, a^*y), \quad \forall \ a \in \A, \ x,y \in \A_0$
	\end{itemize} and  call $\S\in\QA$ an {\em invariant} positive  sesquilinear $\C$-valued map on $\A\times\A$.
\end{definition}	
\begin{definition}\label{defn: rai}  Let $(\A,\A_0)$ be a non-unital quasi *-algebra and $\S:\A\times\A\to\C$ be a positive sesquilinear $\C$-valued map on $(\A,\A_0)$.
A net $\{e_\alpha\}_\alpha\subset\A_0$ is called a right	approximate identity for a  $(\A,\A_0)$ w.r. to  $\S$ if 
    $$\lim\limits_\alpha \S(a-ae_\alpha,a-ae_\alpha)=0,\quad \forall a\in\A.$$
 Further, if $(\A[\|\cdot\|],\A_0)$ is a normed quasi *-algebra, a right	approximate identity $\{e_\alpha\}_\alpha\subset\A_0$ for $(\A,\A_0)$ w.r. to  the norm $\|\cdot\|$ is called {\em strongly idempotent}  
if  $e_\alpha e_\beta=e_\alpha$ whenever $\alpha\leq\beta$.
\end{definition}


 \begin{definition} \label{defn_starrepmod} 
	Let $(\A,\A_0)$ be a non-unital quasi *-algebra.   Let $\DD_\pi$ be a dense subspace
	of a certain quasi $B_\C$-space $\X$ with
	$\C$-valued quasi inner product $\ip{\cdot}{\cdot}_\X$.   A linear map $\pi$ from $\A$ into ${\mathcal L}\ad(\DD_\pi,\X)$ is called  a \emph{*--representation} of  $(\A, \A_0)$,
	if the following properties are fulfilled:
	\begin{itemize}
		\item[(i)]  $\pi(a^*)=\pi(a)^\dagger:=\pi(a)^*\upharpoonright\DD_\pi, \quad \forall \ a\in \A$;
		\item[(ii)] for $a\in \A$ and $x\in \A_0$, $\pi(a)\mult\pi(x)$ is well--defined and {$\pi(a)\mult \pi(x)=\pi(ax)$}.
	\end{itemize}
	
	The *-representation $\pi$  is said to be\begin{itemize}
		\item {\em closable} if there exists $\widetilde{\pi}$ closure of $\pi$ defined as $\widetilde{\pi}(a)=\overline{\pi(a)}\upharpoonright{\widetilde{\D}_\pi}$ where $\widetilde{\D}_\pi$ is the completion under the graph topology $t_\pi$ defined by the seminorms $\xi\in\D_\pi\to\|\xi\|_\X+\|\pi(a)\xi\|_\X$, $a\in\A$, where $\|\cdot\|_\X$ is the norm induced {by the $\C$-valued inner product of $\X$};
		\item \emph{closed} if  $\DD_\pi[t_\pi]$ is complete. 
	\end{itemize}
 \end{definition}

\begin{definition}\label{defn: trans op}
 Let $\{T_n\}_n$ be a sequence of linear operators on a topological vector space $(X,\tau)$  equipped with a topology $\tau$. We say that $\{T_n\}_n$ is topologically transitive on $(X,\tau)$ if, for every two non-empty open subsets $\mathcal{O}_1,\mathcal{O}_2$ of $X$, there exists some $N\in\NN$ such that $T_N(\mathcal{O}_1)\cap\mathcal{O}_2\neq\emptyset.$ \\ A single linear  operator $T$  is topologically transitive if the sequence $\{T^n\}_n$ is topologically transitive.
\end{definition}

\section{Construction of *-representations for positive    $\C$-valued maps on non-unital quasi *-algebras }\label{sec: construction repres}
 In this section we obtain representations for positive   sesquilinear and linear $\C$-valued maps on normed non-unital quasi *-algebras with a  right approximate identity. However, before focusing on the special case of normed non-unital quasi *-algebras, we introduce first the following theorem regarding representation of positive sesquilinear C*-valued maps on more general (i.e. not necessarily normed) non-unital quasi *-algebras. 
\begin{theorem}\label{quasibanach} Let $(\A,\A_0)$ be a non-unital quasi *-algebra  and $\S\in\QA$. Suppose there exists  a net $\{e_\alpha\}_\alpha\subset\A_0$   such that $e_\alpha e_\beta=e_\alpha$ whenever $\alpha\leq\beta$.  Then the following statements are equivalent.
\begin{itemize}
\item[(i)] The net $\{e_\alpha\}_\alpha$ is an approximate identity for $(\A,\A_0)$ w.r. to $\S$. Moreover, for any fixed $e_\alpha$, for every $a\in\A$ there exists a sequence $\{a_n\}_n\subset\A_0$ such that
$\lim\limits_{n\to\infty}\S((a-a_n)e_\alpha,(a-a_n)e_\alpha)=0$.
\item[(ii)] There exist a quasi $B_\C$-space $\X_\S$ with $\C-$valued quasi inner product  $\langle\cdot,\cdot\rangle_{\X_\S}$, a dense subspace $\D_\S\subset \X_\S$, a net $\{\varepsilon_\alpha\}_\alpha\subset\D_\S$ and a closed $*-$representation $\pi_\S$ in  ${\mathcal L}\ad(\D_\S,\X_\S)$ such that for all $a,b\in\A$,
$$\S(a,b)=\lim_\alpha\langle \pi_\S(a)\varepsilon_\alpha,\pi_\S(b)\varepsilon_\alpha\rangle_{\X_\S}.$$
Moreover, $\pi_\S(e_\alpha)\varepsilon_\beta=\varepsilon_\alpha$ whenever $\alpha\leq\beta$ and, for all $a\in\A$ and every $\alpha$, we have that $$\lim_{\beta}\langle \pi_\S(a)(\varepsilon_\beta-\varepsilon_\alpha),\pi_\S(a)(\varepsilon_\beta-\varepsilon_\alpha)\rangle_{\X_\S}$$ exists and 
$$\lim_{\alpha}\lim_{\beta}\langle \pi_\S(a)(\varepsilon_\beta-\varepsilon_\alpha),\pi_\S(a)(\varepsilon_\beta-\varepsilon_\alpha)\rangle_{\X_\S}=0.$$
Finally, given $\varepsilon_\alpha$, for each $a\in\A$ there exists a sequence $\{a_n\}_n\subset\A_0$ such that
$$\pi_\S(a_n)\varepsilon_\alpha\to\pi_\S(a)\varepsilon_\alpha\quad \textrm{in }\X_\S,\quad \mbox{ as } n\to\infty.$$
\end{itemize}
\end{theorem}

\begin{proof}$(i)\Rightarrow(ii):$ Note that $\A_0/\N_\S$ is dense in $\A/\N_\S$. Indeed, given $a\in\A$ and $\delta>0$, by assumption $(i)$ we can find some $e_\alpha$, some sequence $\{a_n\}\subset\A_0$ and a  $N\in\NN$ such that for every $n>N$  we have
$$\|\S(a-ae_\alpha,a-ae_\alpha)\|_\C^{1/2}<\delta/4\quad \mbox{ and }\quad\|\S(ae_\alpha-a_ne_\alpha,ae_\alpha-a_ne_\alpha)\|_\C^{1/2}<\delta/4.$$
Hence,
$$\|a-a_ne_\alpha+\N_\S\|_\S\leq \left(\|a-ae_\alpha+\N_\S\|_\S+\|ae_\alpha-a_ne_\alpha+\N_\S\|_\S\right)<\delta.$$
 Therefore, $\S$ satisfies $(i)$ in \cite[Theorem 3.2]{GBSICT}, hence we can proceed as in \cite[Theorem 3.2]{GBSICT} and construct the designed closed $*-$representation $\pi_\S$. For given $\alpha$, set $\varepsilon_\alpha:=e_\alpha+\N_\S$. Then
\begin{equation}\label{eq: ealpha leq ebeta}
    \pi_\S(e_\alpha)\varepsilon_\beta=e_\alpha e_\beta+\N_\S=e_\alpha+\N_\S=\varepsilon_\alpha,\quad \forall \alpha\leq\beta.
\end{equation}
 Next, since $\langle\cdot,\cdot\rangle_\S$ is jointly continuous, for every $a,b\in\A$ we have
\begin{align*}\S(a,b)&=\langle a+\N_\S,b+\N_\S\rangle_\S=\lim_\alpha\langle ae_\alpha+\N_\S,be_\alpha+\N_\S\rangle_\S\\&=\lim_\alpha\langle \pi_\S(a)\varepsilon_\alpha,\pi_\S(b)\varepsilon_\alpha\rangle_{\X_\S}.    
\end{align*}
Now,  for each $\alpha$ we have 
\begin{align*}
&\S(a-ae_\alpha,a-ae_\alpha)=\lim_\beta\langle\pi_\S(a-ae_\alpha)\varepsilon_\beta,\pi_\S(a-ae_\alpha)\varepsilon_\beta \rangle_{\X_\S}\\
&=\lim_\beta\langle\left(\pi_\S(a)-\pi_\S(a) \mult\pi_\S(e_\alpha)\right)\varepsilon_\beta,\left(\pi_\S(a)-\pi_\S(a) \mult\pi_\S(e_\alpha)\right)\varepsilon_\beta \rangle_{\X_\S}\\
&=\lim_\beta\langle\pi_\S(a)\varepsilon_\beta-\pi_\S(a)\varepsilon_\alpha,\pi_\S(a)\varepsilon_\beta-\pi_\S(a)\varepsilon_\alpha \rangle_{\X_\S},
\end{align*}
where we have used that $\pi_\S(e_\alpha)\varepsilon_\beta=\varepsilon_\alpha$ whenever $\alpha\leq \beta$. It follows that
$$\lim_{\alpha}\lim_{\beta}\langle \pi_\S(a)(\varepsilon_\beta-\varepsilon_\alpha),\pi_\S(a)(\varepsilon_\beta-\varepsilon_\alpha)\rangle_{\X_\S}=\lim_\alpha\S(a-ae_\alpha,a-ae_\alpha)=0.$$ Finally, given $a\in\A$ and some $e_\alpha$, we choose a sequence $\{a_n\}_n\subseteq\A_0$ such that $\lim_{n\to\infty}\S((a-a_n)e_\alpha,(a-a_n)e_\alpha)=0$. Since $\pi_\S(e_\alpha)\varepsilon_\beta=\varepsilon_\alpha$, for $\alpha\leq\beta$, we get
\begin{align*}
&\lim_{n\to\infty}\langle(\pi_\S(a)-\pi_\S(a_n))\varepsilon_\alpha,(\pi_\S(a)-\pi_\S(a_n))\varepsilon_\alpha\rangle_{\X_\S}\\
&=\lim_{n\to\infty}\lim_\beta\langle(\pi_\S(a)-\pi_\S(a_n))\mult\pi_\S(e_\alpha)\varepsilon_\beta,(\pi_\S(a)-\pi_\S(a_n))\mult\pi_\S(e_\alpha)\varepsilon_\beta\rangle_{\X_\S}\\
&=\lim_{n\to\infty}\S(ae_\alpha-a_ne_\alpha,ae_\alpha-a_ne_\alpha)=0,
\end{align*}
which shows that $\pi_\S(\A_0)\varepsilon_\alpha$ is dense in $\pi_\S(\A)\varepsilon_\alpha$, for any $\alpha$.\\

$(ii)\Rightarrow (i):$ Since $\pi_\S$ is a $*-$representation, by hypotheses it is clear that $\S(a,a)\in\C^+$ and $\S(ax,y)=\S(x,a^*y)$ for all $a\in\A$, $x,y\in\A_0$. Next, by the previous calculations, we have 
$$\lim_{\alpha}\S(a-ae_\alpha,a-ae_\alpha)=\lim_\alpha
\lim_\beta\langle\pi_\S(a)(\varepsilon_\beta-\varepsilon_\alpha),\pi_\S(a)(\varepsilon_\beta-\varepsilon_\alpha)\rangle_{\X_\S}=0,$$ where $\pi_\S(e_\alpha)\varepsilon_\beta=\varepsilon_\alpha$ whenever $\alpha\leq \beta$. Finally, for $a\in\A$, since $\pi_\S(\A_0)\varepsilon_\alpha$ is dense in $\pi_\S(\A)\varepsilon_\alpha$, choose a sequence $\{a_n\}\subset\A_0$ such that $$\pi_\S(a_n)\varepsilon_\alpha\to\pi_\S(a)\varepsilon_\alpha,\quad n\to\infty.$$
Then, by the previous calculations we get
\begin{multline*}
  \lim_{n\to\infty}\S((a-a_n)e_\alpha,(a-a_n)e_\alpha)\\=\lim_{n\to\infty}\langle(\pi_S(a)-\pi_\S(a_n))\varepsilon_\alpha,(\pi_S(a)-\pi_\S(a_n))\varepsilon_\alpha\rangle_{\X_\S}=
0.  
\end{multline*}
\end{proof}

If $\S$ satisfies Theorem \ref{quasibanach}, then it will be called *-representable.\\

From now on in this section, we shall consider normed non-unital quasi *-algebras with a right approximate identity. 
\begin{remark}\label{ S bounded + appr id= repres}
	Let $(\A[\|\cdot\|],\A_0)$ be a normed quasi *-algebra, and let $\{e_\alpha\}_\alpha$ be  a right approximate identity for $(\A,\A_0)$.  If $a\in\A$ and $\{a_n\}\subset\A_0$ is such that $a-a_n\to 0$ as $n\to \infty$, then  for each fixed $e_\alpha$ we have that also $(a-a_n)e_\alpha\to 0$ as $n\to \infty$ because the right multiplication by $e_\alpha$ is continuous. Therefore, every bounded $\S\in\QA$ satisfies $(i)$ in Theorem \ref{quasibanach} and hence is *-representable.
	In fact, in this case, it is no longer required that the approximate identity is strongly idempotent. Indeed, by the construction of the $*-$representation from \cite{GBSICT}, combined with the fact that $\varepsilon_\alpha=e_\alpha+\N_\S$, for all $\alpha$,  it follows that 
	$$\S(ae_\alpha-ae_\beta,ae_\alpha-ae_\beta)=\langle \pi_\S(a)(\varepsilon_\alpha-\varepsilon_\beta),\pi_\S(a)(\varepsilon_\alpha-\varepsilon_\beta)\rangle_{\X_\S},$$
	for all $a\in\A$, and all indices $\alpha,\beta$. Hence, if $\S$ is bounded, we deduce that, for all $\alpha$ and all $a\in\A$, we have
	$$\lim_\beta \S(ae_\alpha-ae_\beta,ae_\alpha-ae_\beta)=\S(ae_\alpha-a,ae_\alpha-a)$$ and $$ \lim_\alpha \lim_\beta  \S(ae_\alpha-ae_\beta,ae_\alpha-ae_\beta)=0.$$
	Moreover, $\pi_\S(\A_0)\varepsilon_\alpha$ is dense in $\pi_\S(\A)\varepsilon_\alpha$ for every  $\alpha$. Indeed, let $\{a_n\}_n\subset\A_0$ be such that  $a_n\to a$ in $\A$, then for every $\alpha$ it is $\lim\limits_{\alpha}(a-a_n)e_\alpha=0$ and by the boundedness of $\S$ and the density of $\A_0$ in $\A$ we get
	$$\|\S((a-a_n)e_\alpha,(a-a_n)e_\alpha)\|_\C=\left\|\left(\pi_\S(a)-\pi_\S(a_n)\right)\varepsilon_\alpha\right\|^2_{\X_\S}\to0, \quad\mbox{  as }n\to\infty.$$
	\end{remark}

\begin{corollary}\label{cor: repres omega}
Let $\A$ be a non-unital $C^*-$algebra and $\omega$ be a bounded linear positive map from $\A$ into another $C^*-$algebra $\C$. Then there exist a quasi $B_\C$-space $\X$, a net $\{\varepsilon_\alpha\}_{\alpha}$ in $\X$ and a bounded $*-$representation  $\pi$ in $\B(\X)$ such that 
\begin{equation}\label{eq: first}
    \omega(a)=\lim_\alpha\langle \pi(a)\varepsilon_\alpha,\varepsilon_\alpha\rangle_\X,\quad \forall a\in\A.
\end{equation}
 Moreover,  for all $a\in\A$ and every $\alpha$ the limit $\lim_\beta\langle\pi(a)(\varepsilon_\beta-\varepsilon_\alpha),\varepsilon_\beta-\varepsilon_\alpha\rangle_\X$ exists and we have 
\begin{equation}\label{eq: second}
    \lim_\alpha(\lim_\beta\langle\pi(a)(\varepsilon_\beta-\varepsilon_\alpha),\varepsilon_\beta-\varepsilon_\alpha \rangle_\X)=0.
\end{equation}
\end{corollary}
\begin{proof}
Define $\S:\A\times\A\to\C$ by $\S(a,b)=\omega(b^*a)$. Then $\S$ is a bounded invariant bounded positive sesquilinear $\C-$valued map on $\A\times\A$. Let $\{e_\alpha\}_{\alpha}$ be a right approximate identity for $\A$.  
Since $\A$ is a $C^*-$algebra, take $\A_0:=\A$, hence by Remark \ref{ S bounded + appr id= repres} we deduce that
$$\omega(a^*a)=\S(a,a)=\lim_\alpha\langle\pi_\S(a)\varepsilon_\alpha,\pi_\S(a)\varepsilon_\alpha\rangle_{\X_\S}$$
for all $a\in\A$, where $\varepsilon_\alpha:=e_\alpha+\N_\S$ as before.  Hence, $\omega$ is determined on the cone of positive elements in $\A$. In particular, if $a\in\A^+$, then $a=a^{1/2}a^{1/2}$, so we obtain $$\omega(a)=\lim_\alpha\langle \pi_\S(a^{1/2})\varepsilon_\alpha,\pi_\S(a^{1/2})\varepsilon_\alpha\rangle_{\X_\S}=\lim_\alpha \langle\pi_\S(a)\varepsilon_\alpha,\varepsilon_\alpha\rangle_{\X_\S},$$ because $\pi_\S$ is a $*-$representation. Furthermore, every element of $\A$ can be written as a linear combination of four positive elements in $\A$, hence we deduce that $\omega(a)=\lim_\alpha\langle\pi_\S(a)\varepsilon_\alpha,\varepsilon_\alpha\rangle_{\X_\S}$ for all $a\in\A$. Moreover, since by Remark \ref{ S bounded + appr id= repres}  for all $a\in\A$ and  every $\alpha$ the limit $\lim_{\beta}\langle \pi_\S(a)(\varepsilon_\beta-\varepsilon_\alpha),\pi_\S(a)(\varepsilon_\beta-\varepsilon_\alpha)\rangle_{\X_\S}$ exists and 
$$\lim_{\alpha}\lim_{\beta}\langle \pi_\S(a)(\varepsilon_\beta-\varepsilon_\alpha),\pi_\S(a)(\varepsilon_\beta-\varepsilon_\alpha)\rangle_{\X_\S}=0,$$ by the same arguments as above we get (\ref{eq: second}).\\
 Finally, we have that $\pi_\S(a)$ is a bounded linear operator for all $a\in\A$. To see this, first notice that for all $a,b\in\A$ we have 
$b^*a^*ab\leq\|a\|^2b^*b$. Since $\omega$ is positive we get that
$\omega(b^*a^*ab)\leq\|a\|^2\omega(b^*b)$, which yields the boundedness of $\pi_\S(a)$ for every $a\in\A$:
$$\|\pi_\S(a)(b+\N_\S)\|_{\X_\S}^2=\|\S(ab,ab)\|_\C=\|\omega(b^*a^*ab)\|_\C\leq\|a\|^2\|b+\N_S\|_\S^2, \,\forall b\in\A.$$
 
\end{proof}

\begin{corollary}\label{cor: 3.10}
Let $\A$ be a non-unital $C^*-$algebra and $\omega$ be a bounded linear positive map from $\A$ into another $C^*-$algebra $\C$. Then,
$$4\|\omega\|\,\|\omega(a^*a)\|_\C\geq\|\omega(a)\|_\C^2=\|\omega(a^*)\|_\C\,\|\omega(a)\|_\C, \quad\forall a\in\A.$$
If  $\C$ is commutative, then
$$\|\omega\|\,\|\omega(a^*a)\|_\C\geq\|\omega(a)\|_\C^2=\|\omega(a^*)\|_\C\,\|\omega(a)\|_\C, \quad\forall a\in\A.$$

\end{corollary}
\begin{proof}
By Corollary \ref{cor: repres omega}, for all $a\in\A$, \begin{equation*}    \|\omega(a^*a)\|_\C=\|\lim_\alpha\langle\pi_\S(a)\varepsilon_\alpha,\pi_\S(a)\varepsilon_\alpha\rangle_{\X_\S}\|_\C=\lim_\alpha\|\pi_\S(a)\varepsilon_\alpha\|_\S^2
\end{equation*}since $\pi_S$ is a $*$-representation and the norm is continuous, where $\varepsilon_\alpha:=e_\alpha+\N_\S$, for all $\alpha$, $\{e_\alpha\}_\alpha$ is an approximate identity for $\A$ and $\S(a,b)=\omega(b^*a)$, for all $a,b\in\A$.\\Similarly we have, for all $a\in\A$,
\begin{multline*}    \|\omega(a)\|_\C=\|\lim_\alpha\langle\pi_\S(a)\varepsilon_\alpha,\varepsilon_\alpha\rangle_{\X_\S}\|_{\C}=\|\lim_\alpha\langle\varepsilon_\alpha,\pi_\S(a^*)\varepsilon_\alpha\rangle_{\X_\S}\|_{\C}\\=\|\lim_\alpha\langle\pi_\S(a^*)\varepsilon_\alpha,\varepsilon_\alpha\rangle_{\X_\S}\|_\C=\|\omega(a^*)\|_\C.
\end{multline*}
By  Lemma \ref{lemma 1}, we get
$$4\|\varepsilon_\alpha\|_\S^2\|\pi_\S(a)\varepsilon_\alpha\|_\S^2\geq\|\ip{\pi_\S(a)\varepsilon_\alpha}{\varepsilon_\alpha}_{\X_\S}\|_\C^2, \quad\forall\alpha.$$ 
Moreover, by the construction in the proof of Corollary \ref{cor: repres omega}, we obtain  $\|\varepsilon_\alpha\|_\S^2=\|\omega(e_\alpha^*e_\alpha)\|_\C$. We may let $\{e_\alpha\}_\alpha$ be the canonical approximate identity for $\A$, i.e. $e_\alpha=e_\alpha^*$ and $\|e_\alpha\|\leq1$ and get $\|\varepsilon_\alpha\|_\S^2\leq\|\omega\|$ for all $\alpha$. Therefore we have:\begin{equation}\label{eq: almst}4\|\omega\|\,\|\pi_\S(a)\varepsilon_\alpha\|_\S^2\geq\|\ip{\pi_\S(a)\varepsilon_\alpha}{\varepsilon_\alpha}_{\X_\S}\|_\C^2, \quad\forall \alpha.\end{equation} By taking the limits on both sides of \eqref{eq: almst}, we obtain the desired inequality.\\If  $\C$ is commutative, then we get a better estimate of $\|\omega(a)\|_\C^2$:
$$\|\omega\|\,\|\omega(a^*a)\|_\C\geq\|\omega(a)\|_\C^2=\|\omega(a^*)\|_\C\,\|\omega(a)\|_\C, \quad\forall a\in\A.$$ 
\end{proof}

\begin{corollary}\label{cor:quasibanach} Let $(\A[\|\cdot\|],\A_0)$ be a normed quasi $*-$algebra with a strongly idempotent right approximate identity $\{e_\alpha\}_\alpha\subset\A_0$   and let $\omega$ be a bounded positive linear $\C$-valued map on $(\A,\A_0)$. \ If there exists $M>0$ such that \begin{equation}\label{eq: omega bounded}
    \|\omega(d^*c)\|_\C\leq M\|d\|\|c\|, \quad \forall c,d\in\A_0,
\end{equation} then there exist a quasi $B_\C$-space ${\X_\omega}$ with $\C-$valued quasi inner product  $\langle\cdot,\cdot\rangle_{\X_\omega}$, a dense subspace $\D_\omega\subset {\X_\omega}$ , a net $\{\varepsilon_\alpha\}_\alpha\subset\D_\omega$ and a closed $*$\hyp representation $\pi_\omega$ in { ${\mathcal L}\ad(\D_\omega,{\X_\omega})$} such that for all $a\in\A$ and every $\alpha$ the limit $\lim_\beta\langle \pi_\omega(a)\varepsilon_\alpha,\varepsilon_\beta\rangle_{{\X_\omega}}$ exists and 
$$\omega(a)=\lim_\alpha\lim_\beta\langle \pi_\omega(a)\varepsilon_\alpha,\varepsilon_\beta\rangle_{{\X_\omega}},\quad\forall a\in\A.$$
Moreover, $$\omega(c^*a)=\lim_\alpha\langle \pi_\omega(a)\varepsilon_\alpha,\pi_\omega(c)\varepsilon_\alpha\rangle_{{\X_\omega}},\quad\forall a\in\A, c\in\A_0.$$ In addition, $\pi_\omega(e_\alpha)\varepsilon_\beta=\varepsilon_\alpha$ whenever $\alpha\leq\beta$ and, for all $a\in\A$ and every $\alpha$ the limit  $$\lim_{\beta}\langle \pi_\omega(a)(\varepsilon_\beta-\varepsilon_\alpha),\pi_\omega(a)(\varepsilon_\beta-\varepsilon_\alpha)\rangle_{{\X_\omega}}$$ exists and we have
$\lim_{\alpha}\lim_{\beta}\langle \pi_\omega(a)(\varepsilon_\beta-\varepsilon_\alpha),\pi_\omega(a)(\varepsilon_\beta-\varepsilon_\alpha)\rangle_{{\X_\omega}}=0.$\\
Finally, given $\varepsilon_\alpha$, for each $a\in\A$ there exists a sequence $\{a_n\}_n\subset\A_0$ such that
$$\pi_\omega(a_n)\varepsilon_\alpha\to\pi_\omega(a)\varepsilon_\alpha\quad \textrm{in }{\X_\omega},\quad n\to\infty.$$
\end{corollary}\begin{proof}
Similarly to the proof of \cite[Corollary 3.13]{GBSICT}, we consider the bounded invariant positive sesquilinear $\C$-valued  map $$\phi_0:(c,d)\in\A_0\times\A_0\to\phi_0(c,d)=\omega(d^*c)\in\C$$ which extends, by continuity, to a  bounded  invariant positive sesquilinear $\C$-valued map $\phi$ on $\A\times\A$. For each $e_\alpha$ and $a\in\A$, if $\{c_n\}_n\subset\A_0$ is such that $c_n\to a$ as $n\to\infty$, then $c_n e_\alpha\to a e_\alpha$ 
because the right multiplication operator $R_{e_\alpha}$ is continuous since $(\A,\A_0)$ is a normed quasi *-algebra. Hence $$\omega(ae_\alpha)=\lim_{n\to \infty}\omega(c_ne_\alpha)=\lim_{n\to \infty}\phi_0(e_\alpha,c_n^*)=\phi(e_\alpha,a^*)$$ due to the fact that also $c_n^*\to a^*$ as $n\to\infty$.  By Remark \ref{ S bounded + appr id= repres} applied to $\phi$, we have that \begin{multline*}    \omega(a)=\lim_{\alpha}\omega(ae_\alpha)=\lim_{\alpha}\phi(e_\alpha, a^*)=\lim_{\alpha}\lim_{\beta}\ip{\pi_\phi(e_\alpha)\varepsilon_\beta}{\pi_\phi(a^*)\varepsilon_\beta}_{\X_\phi}\\=\lim_{\alpha}\lim_{\beta}\ip{\varepsilon_\alpha}{\pi_\phi(a^*)\varepsilon_\beta}_{\X_\phi}=\lim_{\alpha}\lim_{\beta}\ip{\pi_\phi(a)\varepsilon_\alpha}{\varepsilon_\beta}_{\X_\phi},
\end{multline*} once $\varepsilon_\alpha=e_\alpha+\N_\phi$ for all $\alpha$ and being $\pi_\phi(e_\alpha)\varepsilon_\beta=\varepsilon_\alpha$, for all $\alpha\leq\beta$.  Moreover, for all $a\in\A$ and $c\in\A_0$, we obtain that \begin{multline*}    \omega(c^*a)=\lim_{n\to \infty}\omega(c^*c_n)=\lim_{n\to\infty}\phi_0(c_n,c)=\phi(a,c)=\lim_{\alpha}\ip{\pi_\phi(a)\varepsilon_\alpha}{\pi_\phi(c)\varepsilon_\alpha}_{\X_\phi}.
\end{multline*}The rest of the thesis is granted by Remark \ref{ S bounded + appr id= repres}.
\end{proof}

\begin{remark}\label{rem: 3.12}
    The class of positive maps $\omega$ on a normed quasi *-algebra $(\A[\|\cdot\|],\A_0)$ which satisfy the assumptions of Corollary \ref{cor:quasibanach} is a proper generalization of the class of bounded positive linear $\C$-valued maps on C*-algebras. Indeed, if $\A=\A_0$ is a C*-algebra and  $\omega$ is a bounded positive linear map $\omega:\A\to\C$, then, for all $c,d\in\A=\A_0$ we have that $$\|\omega(d^*c)\|_\C\leq\|\omega\|\|d^*c\|\leq\|\omega\|\|c\|\|d\|.$$Further, observe that by the triangle inequality applied to the norm $\|\cdot\|_\C$, it follows that if $\omega_1$ and $\omega_2$ satisfy the assumptions of Corollary \ref{cor:quasibanach}, then $\omega_1+\omega_2$ also satisfies the assumption of  Corollary \ref{cor:quasibanach} and, if $\omega_1-\omega_2$ is positive, then  $\omega_1-\omega_2$ satisfies the assumption of  Corollary \ref{cor:quasibanach} as well.\end{remark}

Let  $R_{a}$ denote the right multiplier by $a\in\A$.
\begin{corollary}\label{moreGNS alg}  Let $\A$ be a normed *--algebra 
 and let $\omega$ be a bounded positive  linear $\C$-valued map on $\A$.   Then  the following statements hold.
\\ 

\noindent 1) there
exist a quasi $B_\C$-space $\X_\omega$ whose quasi norm is induced by a $\C$-valued quasi inner product $\ip{\cdot}{\cdot}_{\X}$, a dense subspace $\D_\omega\subseteq\X$ and a closed  *--representation $\pi_\omega$ of $\A$ with domain $\D_\omega$, such that $$\omega(b^*ac)=\ip{\pi_\omega(a) (c+N_\omega))}{b+N_\omega}_{\X_\omega},\quad\forall a,b,c\in\A$$ where $N_\omega=\{a\in\A_0|\, \omega(a^*a)=0\}$.\\ 

     \noindent 2) If $\A$ possesses a right approximate identity $\{e_\alpha\}_\alpha$, there exists a net $\{\varepsilon_\alpha\}_\alpha\subset\X_\omega$  such that for all $a\in\A$ and every $\alpha$ the limit  $\lim_{\beta}\langle \pi_\omega(a)\varepsilon_\alpha,\varepsilon_\beta\rangle_{{\X_\omega}}$ exists and we have
$$\omega(a)=\lim_{\alpha}\lim_{\beta}\langle \pi_\omega(a)\varepsilon_\alpha,\varepsilon_\beta\rangle_{{\X_\omega}}.$$ 

\noindent 3) If, in addition, there exists $M>0$ such that $\|R_{e_\alpha}\|\leq M$ for all $\alpha$, then for all $\alpha$ and $a,b\in\A$ $$\omega(b^*a)=\lim_{\alpha}\langle \pi_\omega(a)\varepsilon_\alpha,\pi_\omega(b)\varepsilon_\alpha\rangle_{{\X_\omega}}.$$ \end{corollary}\begin{proof}
    1) The proof is the same of \cite[Corollary 3.10]{GBSICT} that, in fact, does not require the existence of the unit in $\A$. \\ 
    2) Observe that, since $\A$ is a normed *-algebra, for any fixed $\alpha$ and $a\in\A$, we have that \begin{equation*}
        \|ae_\alpha-e_\beta^*ae_\alpha\|= \|ae_\alpha-(ae_\alpha)^*e_\beta\|,\quad \forall \beta.
    \end{equation*}  Hence, $\lim_\beta\|ae_\alpha-e_\beta^*ae_\alpha\|=0$, thus by the boundedness of $\omega$ we get $$\lim_\beta\omega(e_\beta^*ae_\alpha)=\omega(ae_\alpha).$$ Set  $\varepsilon_\alpha= e_\alpha+\N_\S$ for each $\alpha$, where $\S:\A\times\A\to\C$ is the sesquilinear map defined by $\S(a,b)=\omega(b^*a)$ for all $a,b\in\A$ and $\N_\S$ is the space defined  as in \eqref{eq: NS}. Then by 1), we have $$\lim_{\beta}\omega(e_\beta^*a e_\alpha)=\lim_{\beta}\langle \pi_\omega(a)\varepsilon_\alpha,\varepsilon_\beta\rangle_{{\X_\omega}}=\omega(ae_\alpha).$$ Therefore, $$\omega(a)=\lim_\alpha \omega(ae_\alpha)=\lim_\alpha \lim_{\beta}\langle \pi_\omega(a)\varepsilon_\alpha,\varepsilon_\beta\rangle_{{\X_\omega}}.$$
3) If $\|R_{e_\alpha}\|\leq M $ for all $\alpha$, given $a,b\in\A$ we have \begin{align*}
    \|b^*a-e_\alpha^*b^*ae_\alpha\|&\leq \|b^*a-b^*ae_\alpha\| +\|b^*ae_\alpha-e_\alpha^*b^*ae_\alpha\|\\&\leq \|b^*a-b^*ae_\alpha\| +\|b^*a-e_\alpha^*b^*a\|\|R_{e_\alpha}\|\\&\leq \|b^*a-b^*ae_\alpha\| + M \|a^*b-a^*be_\alpha\|, \quad \forall \alpha
\end{align*}since $\|a^*b-a^*be_\alpha\|=\|b^*a-e_\alpha^*b^*a\|$ because $\A$ is a normed *-algebra. It follows that$$\lim_{\beta}e_\alpha^*b^*ae_\alpha=b^*a, \quad \forall a,b\in\A.$$ Since $\omega$ is bounded, we deduce that \begin{align*}
    \omega(b^*a)&=\lim_\alpha \omega(e_\alpha^*b^*ae_\alpha)=\lim_\alpha\ip{\pi_\omega(a) (e_\alpha+N_\omega)}{be_\alpha+N_\omega}_{\X_\omega}\\&=\lim_\alpha\ip{\pi_\omega(a) \varepsilon_\alpha}{\pi_\omega(b)\varepsilon_\alpha}_{\X_\omega},\quad\forall a,b\in\A.
\end{align*}
\end{proof}

\section{Applications and examples}\label{sec: appl and ex} The aim of this section is to illustrate that examples of non-unital normed quasi *-algebras with strongly idempotent right approximate identity  and  non-trivial positive linear  C*-valued maps satisfying the assumption of Corollary \ref{cor:quasibanach} do exist.  Here we will consider for instance noncommutative generalizations of integral operators.
\bigskip

Let $C(\Omega)$ be the C*-algebra of all continuous functions on a compact Hausdorff space $\Omega$ equipped with the sup norm.
\begin{corollary}
	Let $\A=\ell^2(C(\Omega))$. Then $ \A $ is a normed non-unital $*-$algebra, equipped with the component-wise involution and multiplication. Moreover, it has a strongly idempotent right approximate identity. 
\end{corollary}
\begin{proof}
	The fact that $(\A,\A_0)$ is a normed $*-$ algebra follows directly. For each $m\in\NN$, denote by 
	$$\widetilde{e}_m:=\{\underbrace{1,\ldots,1}_{m-\textrm{times}},0,0,\ldots\}.$$ We have that  $\widetilde{e}_m \widetilde{e}_n=\widetilde{e}_m$ whenever $m\leq n$.   Further, for all $f\in\A$, we have $\|f-f\widetilde{e}_m\|_2\to0$ as $m\to\infty$. 
\end{proof}

 In particular,  statement 3) in Corollary \ref{moreGNS alg} applies to the case of  $ \ell^2(C(\Omega))$, since it is a normed *-algebra whose right approximate identity $\{\widetilde{e}_m\}_m$ is such that $\|R_{\widetilde{e}_m}\|\leq1$ for all $m$.

\begin{example}Let $\{w_n\}_n\subseteq C(\Omega)$ be a uniformly bounded sequence of positive functions  i.e. $|w_n(x)|\leq M$ $\forall x\in\Omega$,  $\forall n\in\NN$ for some $M>0$ and $\S:\ell^2(\Omega)\times\ell^2(\Omega)\to C(\Omega)$ be given by $$\S((f_1,f_2,...), (g_1,g_2,...))=\sum_{n=1}^\infty w_nf_n\overline{g_n}, \quad (f_1,f_2,...),(g_1,g_2,...)\in \ell^2(\Omega).$$
	It is straightforward to check that $\S$ is a bounded invariant positive sesquilinear  $C(\Omega)$-valued map on $\ell^2(\Omega)\times\ell^2(\Omega)$. Then it is *-representable.\\
 Now, let $\{w_n\}\in\ell_2(C(\Omega))$ with $w_n\geq0$ for every $n\in\NN$. Since $w_j\leq\left(\sum_{n=1}^\infty w_n^2\right)^{\frac{1}{2}}$, it is $\|w_j\|_\infty\leq\|\{w_n\}_n\|_2$,  for all $j\in\NN$. Define $$\omega:\ell_2(C(\Omega))\to C(\Omega)$$ by $$\omega(f_1,f_2,...)=\sum_{n=1}^\infty w_nf_n, \quad\forall (f_1,f_2,...)\in\ell_2(C(\Omega)).$$ Then, since $\left\|\sum_{n=1}^\infty w_nf_n\right\|_\infty\leq\|\{w_n\}_n\|_2\, \|\{f_n\}_n\|_2$, $\omega$ is  a bounded positive linear  $C(\Omega)$-valued map. Moreover, for every $ \{f_n\}_n,\{g_n\}_n\in \ell_2(C(\Omega))$, we have that $$ \|\omega(\overline{g_1}f_1, \overline{g_2}f_2, ...)\|_\infty\leq\|\{w_n\}_n\|_2\,\|\{f_n\}_n\|_2\,\|\{g_n\}_n\|_2,$$ hence $\omega$ satisfies the assumptions of the Corollary \ref{cor:quasibanach}.
\end{example}

\begin{remark}
	Consider the normed non-unital quasi $*-$algebra $(L^2(\RR),L^\infty_c(\RR))$, with $L^\infty_c(\RR)$ the $*-$algebra of all bounded measurable functions in $\RR$ with compact support. It is not hard to see that, if $e_n=\chi_{[-n,n]}$, for each $n\in\NN$, i.e. the characteristic function of the interval $[-n,n]$, then 
	$e_m e_n=e_m$ whenever $m\leq n$. Further, the sequence $\{e_n\}_n$ is a strongly idempotent approximate identity for $(L^2(\RR),L^\infty_c(\RR))$. 
\end{remark}

\subsection{Integral Operators}
In order to build our way up to the abstract setting, we start this subsection with the following elementary example.  
\begin{example}\label{l2k} Let $L^\infty_c(\RR)$ be as before and consider the quasi $*-$algebra $(L^2(\RR),L^\infty_c(\RR))$. For all $n\in\NN$, set $e_n:=\chi_{[-n,n]}$ and $\S: L^2(\RR)\times L^2(\RR)\to\CC$, defined as $$\S(f,g):=\int_{\RR}f(t)\overline{g(t)}v(t)dt,$$
	where $v$ is some bounded non-negative measurable function on $\RR$. It is not hard to see that $\S$ satisfies $(i)$ in Theorem \ref{quasibanach}, hence it is *-representable.
\end{example}

More generally, let $C_b(\Omega)$ be the space of bounded continuous functions on $\Omega$, $k\in C_b(\RR^2)$ and let $\S_k:L^2(\RR)\times L^2(\RR)\to  C_b(\RR)$, be given by 
$$\S_k(f,g)(x):=\int_{\RR}k(x,t)f(t)\overline{g(t)}dt,$$
for $x\in\RR$, and $f,g\in L^2(\RR)$.  Then $\S_k$ satisfies $(i)$ of Theorem \ref{quasibanach}.  Indeed, it is not hard to check by some calculations that the image of $\S_k$ is in fact a subset of $C_b(\RR)$. Moreover, if $w\in L^2(\RR)\cap  L^\infty(\RR)$ and $w\geq0$, then the map $\theta:L^2(\RR)\to C_b(\RR) $  defined by $$\theta(f)(x)=\int_{\RR}k(x,t)w(t)f(t)dt,\quad f\in L^2(\RR), x\in\RR,$$ is a well-defined, bounded positive linear map on $L^2(\RR)$ satisfying the conditions of Corollary \ref{cor:quasibanach}.\\ 

This can be even more generalized, as we will see in what follows. Before introducing the result, we have to establish the notation. Denote by $L^2(\RR; \C)$ the Banach space of square integrable $\C-$valued functions, where the integral is regarded as the Bochner integral with respect to the Lebesgue measure $\mu$ on $\RR$:
$$f\in L^2(\RR,\C)\Leftrightarrow \int_\RR \|f(t)\|_{\C}^2\,d\mu(t)<\infty,$$
where we equate the functions which are equal $\mu-$almost everywhere; see \cite{DU}, \cite{DS}, or \cite{JM} for details on Bochner integrals. Denote by $\B(\C)$ the unital Banach algebra of bounded linear operators over $\C$, and by $C_b(\RR^2;\B(\C))$ and $C_b(\RR^2;\C)$ the spaces of the uniformly bounded mappings respectively from $\RR^2$ to  $\B(\C)$ and from $\RR^2$ to  $\C$.

\begin{corollary}	Let $K\in C_b(\RR^2;\B(\C))$ and define the sesquilinear form $\S_K: L^2(\RR;\C)\times L^2(\RR;\C)\to C_b(\RR;\C)$ by
	\begin{equation}\label{SK}\S_K(f,g)(x):=\int_\RR K(x,t) f(t)g^*(t)d\mu(t),\quad f,g\in L^2(\RR; \C).\end{equation} Then the following statements hold.
	\begin{itemize}
		\item[1)] For every $K\in C_b(\RR^2;\B(\C))$ and for every  $f,g\in L^2(\RR;\C)$, the mapping $\S_K(f,g):\RR\to \C$ is bounded and  continuous. Moreover, if $K$ is differentiable in the first coordinate, with the bounded and continuous partial derivative $\partial_x K(x,t)\in C_b(\RR^2;\B(\C))$, then the function $x\mapsto S_K(f,g)(x)$ is differentiable in $x$, and 
		\begin{equation}\label{SK'}\left(\S_K(f,g)\right)'(x)=\int_\RR \,\partial_x K(x,t) f(t)g^*(t)d\mu(t).\end{equation}
		\item[2)] If $K\in C_b(\RR^2;\B(\C))$ is such  that $K(x,t)ff^*\in\C^+$, for all $x,t\in\RR$, whenever $f$ is a fixed element in $L^2(\RR;\C)$, then, $\S_K$ is *\hyp representable. 
	\end{itemize}
\end{corollary}

\begin{proof}1) Since $K\in C_b(\RR^2;\B(\C))$ is uniformly bounded and $f,g\in L^2(\RR;\C)$, we have, for every $x\in\RR$
		\begin{multline*}
			\left\|\S_K(f,g)(x)\right\|_\C\leq\int_\RR\|K(x,t)\|\,\|f(t)g^*(t)\|_{\C}\,d\mu(t)\\\leq\sup_{t\in\RR}\|K(x,t)\|\cdot\|fg\|_1<\infty,
		\end{multline*}
		so the function $x\mapsto \S_K(f,g)(x)$
		is uniformly bounded, for every choice of $K,f,g$. To prove its continuity, it suffices to observe that
		for any $x,y\in\RR$ we have:
		\begin{multline*}
			\left\|\S_K(f,g)(x)-\S_K(f,g)(y)\right\|_\C\\\leq
			\int_\RR\|K(x,t)-K(y,t)\|\,\|f(t)g^*(t)\|_{\C}\,d\mu(t)\\
			\leq\sup_{t\in\RR}\|K(x,t)-K(y,t)\|\,\|fg\|_1\to 0,\quad\mbox{as } |x-y|\to0.
		\end{multline*}
		
		Furthermore, let $f,g\in L^2(\RR;\C)$ be fixed. Then, by the Radon-Nykodim Theorem for the Bochner integral, there exists a unique finite $\C-$valued measure, $\nu(f,g)$, defined as
		$$\nu(f,g)(E):=\int_E f(t) g^*(t)d\mu(t),$$
		for every Borel subset $E$ of $\RR$. Therefore, for any given $K\in C_b(\RR^2; \B(\C))$, the mapping
		$$x\mapsto S_K(f,g)(x)=\int_\RR K(x,t)d\nu(f,g)(t)$$
		is continuous. If, in addition, the mapping $K$ is differentiable in its first argument, with  continuous bounded partial derivative, then
		$$\begin{aligned}&\lim_{\Delta x\to 0}\frac{1}{\Delta x}(\S_K(f,g)(x+\Delta x)-\S_K(f,g)(x))\\
			&=\lim_{\Delta x\to 0}\frac{1}{\Delta x}\left(\int_{\RR}K(x+\Delta x,t) d\nu(f,g)(t)-\int_{\RR}K(x,t)d \nu(f,g)(t)\right)\\
			&=\lim_{\Delta x\to 0}\frac{1}{\Delta x}\int_{\RR}(K(x+\Delta x,t)-K(x,t)) d\nu(f,g)(t)\\
			&=\int_{\RR}\partial_x K(x,t) d\nu(f,g)(t).
		\end{aligned}$$
	2) By assuming that $K(x,t)ff^*\in\C^+$, for every $f\in L^2(\RR;\C)$ which is definite $\mu-$almost everywhere,  we prove that the sesquilinear form $\S_K$ is positive. Then, similarly to Example \ref{l2k}, by taking $\{e_n\}_n$ with $e_n:=\chi_{[-n,n]}$ as an approximate identity with respect to $\S_K$, it is possible to show that $\S_K$ satisfies $(i)$ in Theorem \ref{quasibanach}, hence it is *-representable. 
\end{proof}

Keeping the notation and the assumptions from the previous corollary, consider the functional
$$\omega_K(cd)(x)=\int_\RR K(x,t)c(t)d(t)dt$$
for every $c,d\in L_2(\RR;\C)$. Specially, if $K(x,t)$ is a positive operator in $\B(\C)$, we have
$$\omega_K(ff^*)(x)=\int_\RR K(x,t)f(t)f^*(t)dt,$$
thus $\omega_K$ is a positive map with values in $C_b(\RR;\C)$. Indeed, the boundedness follows from
$$\left\|\omega_K(ff^*)\right\|\leq \sup_{t\in\RR}\|K(x,t)\|\cdot\|ff^*\|_1= \sup_{t\in\RR}\|K(x,t)\|\cdot\|f\|_2^2.$$
Consequently, the mapping $\omega_K$ is a positive bounded map with values in the unital algebra $C_b(\RR;\C)$.

Denote by $C(\RR;\B(\C))$ the space of continuous $\B(\C)-$valued functions with real arguments. Recall that the algebra $C(\RR;\B(\C))$, equipped with the pointwise multiplication and involution, is a unital *-algebra. However, the algebra of $\B(\C)-$valued functions which vanish at infinity, denoted as $C_0(\RR;\B(\C))$, equipped with the operations and topology inherited from $C(\RR;\B(\C))$, is an example of a non-unital $*-$algebra. 

In that sense, assume that $\sup_{t\in\RR}\|K(\cdot,t)\|$ belongs to $C_0(\RR;\B(\C))$, i.e.
$$\lim_{|x|\to\infty}\sup_{t\in\RR}\|K(x,t)\|=\mathbf{0}_{C(\RR;\B(\C))}.$$
Then, due to 
$$\left\|\omega_K(ff^*)\right\|\leq \sup_{t\in\RR}\|K(x,t)\|\cdot\|f\|_2^2,$$
the mapping $\omega_K$ has its range contained in the non-unital algebra $C_0(\RR;\B(\C))$. Embed the image of $\omega_K$ into $C_0(\RR;\B(\C))$ via the inclusion operator $\imath$ in the way that $$\imath:\R(\omega_K)\hookrightarrow C_0(\RR;\B(\C)),\quad \imath(u)=u\in C_0(\RR;\B(\C)),$$ 
and denote the composition $\omega_K\circ\imath$ as $\omega_K$ again:
$$\omega_K(cd)\in C_0(\RR;\B(\C)),\quad c,d\in L^2(\RR;\C).$$ This way, the mapping $\omega_K$ is subject to the Corollary \ref{cor:quasibanach}.

\subsection{Noncommutative $L^p$-spaces and generalized integral operators}
The previous section gives a general idea on how to approach the integral operators in the noncommutative non-unital setting. In what follows we further explore this class of operators. 

Let $\H$ be a separable Hilbert space with an orthonormal basis  $\{e_j\}_{j}$ , and let $B_p(\H)$ be the Schatten $p-$ideal with the $p-$norm \cite{GK}. Then, $B_p(\H)$ is a Banach space. For each $m\in\NN$, let $P_m$ denote the orthogonal projection onto the span of $\{e_1,\ldots,e_m\}$. Then, $P_m$ is a finite rank operator for every $m$. Notice that banally it is $P_mP_n=P_m$ for every $m\leq n$.  In a similar way as in \cite[Remark 2.9]{IT2} one can prove the following lemma. 

\begin{lemma}\label{cor: app identity} 
 The normed non-unital  $*-$algebra $B_p(\H) $ has a strongly idempotent right approximate identity  $\{P_m\}_m$  consisting of orthogonal projections.
\end{lemma}

 In particular, the statement 3) in Corollary \ref{moreGNS alg} applies to the case of   $(B_p(\H),\|\cdot\|_p)$, since it is a normed *-algebra whose right approximate identity $\{P_m\}_m$ is such that $\|R_{P_m}\|\leq1$ for all $m$.

\begin{example} \label{schatten}
    Let $p>2$. Choose some $\{\lambda_j\}_{j}\in \ell_{\frac{p}{p-1}}\cap\ell_{\frac{p}{p-2}}$ with $\lambda_j\geq0$ for all $j\in\NN$ and such that  $ \lambda_i> \lambda_j$ for $i< j$. Let  $\{g_n\}_n$ be a uniformly bounded sequence of positive continuous functions on $[0,1]$ such that $ g_i(t)\geq g_j(t)$ for $i\leq j$ and  for all $t\in[0,1]$.     Define $$W_t=\sum_{j=1}^\infty \lambda_jg_j(t)\ip{\cdot}{e_j}e_j, \quad t\in[0,1]$$  where $\{e_j\}_j$ denotes an orthonormal basis of a separable Hilbert space $\H$. Let $$\omega(A)(t)=tr(AW_t),\quad \forall A\in\B_p(\H),\, t\in[0,1].$$ By some calculations it is not hard to check that $\omega$ is a bounded positive linear map from $ B_p(\H) $ into $ C([0,1]) $ that satisfies the assumptions of the Corollary \ref{cor:quasibanach}. 
\end{example}

Let now $\M$ be a von Neumann algebra
and $\rho$ a normal semifinite faithful trace on $\M$. Denote by $L^p(\rho)$ the Banach space consisting of operators
affiliated with $\M$ which is the completion of the
*-ideal
$$\mathcal{J}_p :=
\{X \in\M : \rho(|X|^p) < \infty\},$$
with respect to the norm $\|X\|_p := \rho(|X|^p)^{1/p}, X\in\ \M$ (see \cite[Example 3.1.7]{MFCT} and references therein).\\
\begin{lemma}\label{projectors} Let $\M$ be a von Neumann algebra which is a factor of type I or II, and $\rho$ be a semifinite trace on $\M$. Let  $ W\in\M$ such that $W\geq0$. Then there exists a sequence $\{P_n\}_n$ of finite projections in $\M$ such that 
$$\lim_{n\to\infty}\|W(I-P_n)\|_p=0.$$
\end{lemma}
\begin{proof}
Suppose that $\M$ is a factor of type I or II, and that $\rho$ is a semifinite trace on $\M^+$. It is known that $\rho$ is unique up to a scalar multiplication. Moreover, for each positive $T\in\M$ it follows that
$$\rho(T)=\int_0^\infty \lambda d(D\circ E_T),$$
where $E_T$ is the spectral measure corresponding to $T$ and $D$ is the dimension function, which is also unique up to a scalar multiplication. Given $p \in \mathbb{N}$, define $\alpha:\RR\to\RR$ by $\alpha(t)=t^p$.  If  $W\in\M$ is such that $W\geq0$, then it is $E_{W^p}=E_W\circ\alpha^{-1}$, and for every Borel subset $B\subset\RR$,  $$E_{(WE_{W}([0,1/n]))^p}(B)=E_W(\alpha^{-1}( B) \cap [0,1/n]), \quad  n\in\NN.$$
Consequently, 
$$\rho(W^p)=\int_0^\infty \lambda^p d(D\circ E_W)$$
and
$$\rho((WE_{W}([0,1/n]))^p)=\int_0^{1/n}\lambda^p d(D\circ E_W).$$
If $\rho(W^p)<\infty$ then by the Monotone Convergence Theorem 
$$\int_0^\infty\lambda^p d(D\circ E_W)=\lim_{n\to\infty}\int_{1/n}^\infty \lambda^p d(D\circ E_W).$$
Furthermore, we have that
$$\frac{1}{n^p}\rho(E_W(1/n,\infty))=\frac{1}{n^p}(D\circ E_W((1/n,\infty)))\leq\int_{1/n}^\infty \lambda^p d(D\circ E_W)<\infty,$$
so $E_W((1/n,\infty))$ is finite. By letting $P_n=E_W\left(1/n, \infty\right)$ for each $n\in\NN$, we deduce the thesis.\end{proof}
Now we will consider the applications to $L^2(\rho)$. First recall that, if $T\in L^2(\rho)\cap L^\infty(\rho)$ and $B\in L^\infty(\rho)$, then $TB, BT\in L^2(\rho)$, and
$\|TB\|_2$, $\|BT\|_2\leq \|B\|_\infty\cdot\|T\|_2$. 
We obtain the following corollary:
\begin{corollary}\label{cor: app identity 2}

Let $\M$ be a von Neumann algebra and $\rho$ a semifinite (non finite) trace on $\M$. The normed non-unital quasi $*-$algebra $(L^2(\rho), L^2(\rho)\cap L^\infty(\rho))$ has a strongly idempotent right approximate identity $\{P_\alpha\}_\alpha$ consisting of finite projections in $\M$.
\end{corollary}
\begin{proof} Let $F\in L^2(\rho)\cap L^\infty(\rho)$. By Lemma \ref{projectors}  for every $\epsilon>0$ there exists a  finite projection $P_\epsilon$ such that $\| |F|(I-P_\epsilon)\|_2<\epsilon.$ It is also $$\| F(I-P_\epsilon)\|_2=\| U|F|(I-P_\epsilon)\|_2\leq \|U\|_\infty\| |F|(I-P_\epsilon)\|_2<\epsilon,$$ where $U$ is the partial isometry from the polar decomposition of $F$. Now, if $Q$ is any finite projection with $P_\epsilon\leq Q$, then, since $F(I-Q)F^*\leq F(I-P_\varepsilon)F^*$, we get $$\|F(I-Q)\|_2=\|(I-Q)F^*\|_2\leq\|(I-P_\varepsilon)F^*\|_2=\|F(I-P_\varepsilon)\|_2<\epsilon.$$
Let  $\{P_\alpha\}_\alpha$ denote the net of finite projections in $L^\infty(\rho)$, then  $P_\alpha P_\beta=P_\alpha$ for every $\alpha\leq \beta$. It is  $$\lim_\alpha\|F(I-P_\alpha)\|_2=0, \quad\forall F\in L^2(\rho)\cap L^\infty(\rho).$$  Given $T\in L^2(\rho)$  and $\epsilon>0$,  there exists a $F\in L^2(\rho)\cap L^\infty(\rho)$ such that  $\|T-F\|_2<\frac{\epsilon}{3}$. Moreover, by the previous arguments, there exists a finite projection $P_\epsilon$, such that $\|F(I-Q)\|_2<\frac{\epsilon}{3}$ for every finite projection $Q$ with  $P_\epsilon\leq Q$. Hence we deduce that \begin{multline*}
    \|T(I-Q)\|_2\leq\|T-F\|_2+\|F(I-Q)\|_2+\|(F-T)Q\|_2\\\leq \|T-F\|_2+\|F(I-Q)\|_2+\|T-F\|_2\|Q\|_\infty<\epsilon.
\end{multline*} for every finite projection $Q$ with  $P_\epsilon\leq Q$, hence  in particular $$\lim_\alpha\|T(I-P_\alpha)\|_2=0.$$ 
\end{proof}

\begin{example}\label{ex: 4.11}
    Consider the $(L^2(\rho), L^\infty(\rho)\cap L^2(\rho))$, where $\rho$ is semi-finite trace. Let $W\in L^\infty(\rho)\cap L^2(\rho)$ such that $W\geq0$. By Lemma \ref{projectors}, there exists a finite projection $P$ such that $PW=WP$. Then, $WP=PWP\geq0$. Let $k$ be a nonnegative continuous function on $[0,\|WP\|]\times[0,\|WP\|]$. For each $x\in[0,\|WP\|]$, let $f_x\in C([0,\|WP\|])$  be given by $f_x(t)=k(x,t)$ and $W_x:=f_x(W P)$. Define $\omega$ on $L^2(\rho)$ by $$\omega(A)(x)=\rho(A((I-P)W+W_x)), \quad \forall x\in[0,\|WP\|], A\in L^2(\rho).$$ By some calculations, it is not hard to check that $ \omega $ is a bounded, positive linear map from $ L^2(\rho) $ into $ C([0,\|WP\|])$ that satisfies the assumption in Corollary \ref{cor:quasibanach}. Observe that $\omega$ induces also a bounded  positive sesquilinear $C([0,\|WP\|])$-valued map on $L^2(\rho)\times L^2(\rho)$ given by $$\phi(X,Y)(x)=\rho(X((I-P)W+W_x)Y^*),\quad x\in [0,\|WP\|],$$ for all $X,Y\in L^2(\rho)$.\\
     Now, let $L^2([0,\|WP\|],\B(\H))$  with respect to the Gel'fand-Pettis integral (see \cite{Jocic}) and choose  $A_t\in L^2([0,\|WP\|],\B(\H))$
  such that $A_t\geq0$ for a.e. $t\in [0,\|WP\|]$. Similarly as earlier, we can construct a bounded positive linear map $\theta:L^2(\rho) \to \B(\H)$ with $$\theta(X)=\int_0^{\|WP\|}\rho(X((I-P)W+W_t))A_t \,dt, \quad X\in L^2(\rho).$$ Finally, put $\B^a(\ell_2(C([0,\|WP\|])))$ the C*-algebra of all adjointable bounded operators on $\ell_2(C([0,\|WP\|]))$ which are linear w.r. to $C([0,\|WP\|])$ and define $$\widetilde{\theta}:L^2(\rho)\to \B^a(\ell_2(C([0,\|WP\|])))$$ by $$\widetilde{\theta}(X)(f_1,f_2,...)=(\omega(X)f_1,\omega(X)f_2,...), $$ for every $ X\in L^2(\rho)$, $(f_1,f_2,...)\in \ell_2(C([0,\|WP\|]))$. Then both $\theta$ and $\widetilde{\theta}$ satisfy the assumptions of Corollary \ref{cor:quasibanach}. Notice that in a similar way as for $\omega$, the maps $\theta$ and $\widetilde{\theta}$ induce bounded  invariant positive sesquilinear maps on $L^2(\rho)\times L^2(\rho)$ that take values in $\B(\H)$ and $\B^a(\ell_2(C([0,\|WP\|])))$, respectively.
\end{example}

 \begin{remark}\label{genint}
    If $T:L^2(\RR)\to C([0,1])$ is an integral operator given by $$T(f)(x)=\int_\RR k(x,t)f(t)dt, \quad \forall x\in[0,1],\, f\in L^2(\RR)$$where $k(x,\cdot)\in L^2(\RR)$ for all $x\in[0,1]$, then $k$ induces a mapping $\eta:[0,1]\to L^2(\RR)$ given by $$\eta(x)=k(x,\cdot), \quad \forall x\in[0,1].$$ Moreover, $$T(f)(x)=\int_\RR \eta(x)(t)f(t)dt, \quad \forall x\in[0,1],\, f\in L^2(\RR).$$Hence, if $\widetilde{\eta}$ is a mapping from $[0,1]$ into $L^2(\rho)$ and $\widetilde{T}$ is an operator on $L^2(\rho)$ given by $$\widetilde{T}(Y)(x)=\rho(Y\widetilde{\eta}(x)),\quad \forall x\in[0,1],\, Y\in L^2(\rho),$$ then, if $\widetilde{T}(Y)(\cdot)$ belongs to $C([0,1])$ for every $Y\in L^2(\rho)$, $\widetilde{T}$ can be  considered as a generalized integral operator on $L^2(\rho)$. Thus, the map $\omega$ in Example \ref{ex: 4.11} can be in this way interpreted as a generalized integral operator on $L^2(\rho)$. A similar consideration applies to the positive map constructed in Example \ref{schatten}.  \end{remark}

\section{Topologically transitive operators on noncommutative $L^2$-spaces }\label{sec: Inf VN algebras} Let $\M$ be a properly infinite von Neumann algebra over a separable Hilbert space $\H$. Then, $\H$ can be written as $\H=\oplus_{j\in\ZZ} \H_j$, where for each $j$ the orthogonal projection onto $\H_j$, denoted as $P_j$, belongs to $\M$, and $P_{j_1}\sim P_{j_2}$, for all $j_1,j_2\in\ZZ$. Let $\{W_j\}_{j\in\ZZ}$ be a uniformly bounded sequence in $\M$, and set \begin{equation}\label{eq: W}W:=\sum_{j\in\ZZ}\lambda_jP_jW_jP_j,\end{equation} where $\lambda_j\geq0$ for all $j$, and $\lim\limits_{|j|\to\infty} \lambda_j=0$. We have the following 
\begin{lemma}\label{lem: 5.1}
The above defined operator $W$ is a well-defined bounded linear operator which belongs to $\M$.
\end{lemma}
\begin{proof}
Since $\{W_j\}_{j\in\ZZ}$ is uniformly bounded, there exists some $M>0$ such that $\|W_j\|\leq M$ for all $j\in\ZZ$. If $x\in\H$, then $\|x\|^2=\sum\limits_{j\in\ZZ}\|P_jx\|^2$. Hence, given $\varepsilon>0$, there exists some $N\in\NN$ such that
$$\sum_{j=N}^\infty \|P_jx\|^2<\frac{\varepsilon}{2M^2R^2},\quad \textrm{and}\quad\sum_{j=N}^{\infty}\|P_{-j}x\|^2<\frac{\varepsilon}{2M^2R^2},$$
where $R=\sup_{j\in\ZZ}\{\lambda_j\}$. Thus, for every $n,m\in\NN$ such that $m>n\geq N$, we obtain
$$\begin{aligned}
\left\|\sum_{j=n}^m \lambda_jP_jW_jP_jx\right
\|^2&=\sum_{j=n}^m\lambda_j^2\|P_jW_jP_jx\|^2\\
&\leq R^2\sum_{j=n}^m\|P_jW_j\|^2\|P_jx\|^2\\
&\leq M^2R^2\sum_{j=n}^m\|P_jx\|^2\\
&<\varepsilon/2.\end{aligned}$$
Similarly, we get
$$\left\|\sum_{j=n}^m \lambda_{-j}P_{-j}W_{-j}P_{-j}x\right\|^2<\varepsilon/2.$$
Hence, the sums $\sum\limits_{j=1}^\infty \lambda_jP_jW_jP_jx$, and $\sum\limits_{j=1}^\infty \lambda_{-j}P_{-j}W_{-j}P_{-j}x$ are convergent in $\H$. Since $x\in\H$ was chosen arbitrarily, we deduce that
$$s-\lim_{n\to\infty}\sum_{j=1}^n \lambda_jP_jW_jP_j\in \M,\quad s-\lim_{n\to\infty}\sum_{j=1}^n \lambda_{-j}P_{-j}W_{-j}P_{-j}\in \M,$$
so $W$ belongs to $\M$ as well.
\end{proof}

If $W_j\geq0$ for every $j\in\ZZ$ and the assumptions of Lemma \ref{lem: 5.1} hold,  then $W$ is a bounded positive linear operator. Moreover, since $\lim\limits_{|j|\to\infty}\lambda_j=0$, we get $$\lim\limits_{n\to\infty}\left\|W(I-\sum_{j=-n}^n P_j)\right\|=0.$$
Indeed, for each $x\in\H$ and every $n\in\NN$, we have 
$$\begin{aligned}
\left\|W(I-\sum_{j=-n}^n P_j)x\right\|^2&=\sum_{j={n+1}}^\infty \lambda_j^2\|P_jW_jP_j x\|^2+\sum_{j={n+1}}^\infty \lambda_{-j}^2\|P_{-j}W_{-j}P_{-j}x\|^2\\
&\leq \sup_{j>n}\lambda_j^2 M^2\sum_{j=n+1}^\infty \|P_jx\|^2+\sup_{j>n}\lambda_{-j}^2 M^2\sum_{j=n+1}^\infty\|P_{-j}x\|^2\\
&\leq\sup_{|j|>n}\lambda_j^2 M^2\left(\sum_{j\in\ZZ}\|P_jx\|^2\right)=\sup_{|j|>n}\lambda^2_j M^2\|x\|^2.
\end{aligned}$$
Next, since $P_j\sim P_{j+1}$ for all $j\in\ZZ$, there exists for each $j$ a partial isometry $U_j\in \M$, that maps $\H_j$ isometrically onto $\H_{j+1}$.  Given $x\in\H$ and $\varepsilon>0$, choose again some $N>0$ such that $\sum_{j=N}^\infty\|P_jx\|^2<\varepsilon/2$ and $\sum_{j=N}^\infty \|P_{-j}x\|^2<\varepsilon/2$. For each $m>n>N$, we have that 
$$\left\|\sum_{j=n}^m P_{j+1} U_j P_j x\right\|^2\leq\sum_{j=n}^\infty\|P_jx\|^2<\varepsilon/2,$$
and, similarly, 
$$\left\|\sum_{j=n}^mP_{-j+1}U_{-j}P_{-j}x\right\|^2<\varepsilon/2.$$
For an argument similar to the one above, we deduce that 
$$s-\lim_{n\to\infty} \sum_{j=0}^n P_{j+1}U_jP_j\in \M,\quad s-\lim_{n\to\infty} \sum_{j=1}^n P_{-j+1}U_{-j}P_{-j}\in \M.$$
Set $$V=2\left( s-\lim_{n\to\infty} \sum_{j=1}^n P_{-j+1}U_{-j}P_{-j}\right)+\frac{1}{2}\left(s-\lim_{n\to\infty} \sum_{j=0}^n P_{j+1}U_jP_j\right).$$
It is straightforward to check that $V$ is invertible and 
$$V^{-1}=\frac{1}{2}\left(s-\lim_{n\to\infty} \sum_{j=1}^n P_{-j}U^*_{-j}P_{-j+1}\right)+2\left(s-\lim_{n\to\infty} \sum_{j=0}^n P_jU^*_jP_{j+1}\right).$$
For each $j\in\NN$, and all $n>j$, one can check that
$$V^nP_{-j}=2^{2j-n}\left(U_{-j+n-1}\cdot\ldots\cdot U_{-j+1}U_{-j}P_{-j}\right),$$ 
while for all $n\in\NN$, we have that
$$V^nP_j=2^{-n}\left(U_{j+n-1}\cdot\ldots\cdot U_{j+1}U_jP_j\right).$$
The latter also holds when $j=0$. Hence we get 
$$\left\|V^n\left(\sum_{j=-k}^k P_j\right)\right\|\leq\sum_{j=-k}^k\|V^nP_j\|\leq 2k2^{2k-n},$$
whenever $n>k$ and $k\in\NN$, so 
$$\lim_{n\to\infty}\left\|V^n\left(\sum_{j=-k}^k P_j\right)\right\|=0,\quad \forall k\in\NN.$$
By similar calculations, one can check that
$$\lim_{n\to\infty}\left\|V^{-n}\left(\sum_{j=-k}^k P_j\right)\right\|=0,\quad \forall k\in\NN.$$

Let $R_{V^*}: L^2(\rho)\to L^2(\rho)$ be given by $R_{V^*}(A):=AV^*$, for all $A\in L^2(\rho)$. Since $\|AV^*\|_2\leq \|A\|_2\|V^*\|_\infty$ for all $A\in L^2(\rho)$, it follows that  $R_{V^*}$ is a well-defined bounded linear operator on $L^2(\rho)$. Consider the positive sesquilinear form $\S: L^2(\rho)\times L^2(\rho)\to\C$, given by 
$$\S_W(X,Y):=\rho(XWY^*),\quad X,Y\in L^2(\rho),$$
where $W$ is the positive linear operator in $\M$, introduced in \eqref{eq: W}. Since $\S_W$ is positive and sesquilinear, it induces a semi-norm $\|\cdot\|_W$ on $L^2(\rho)$, given by $\|X\|_W:=(\S_W(X,X))^{1/2}$ for all $X\in L^2(\rho)$. Let $(L^2(\rho),\tau_W)$  be the locally convex topological vector space, equipped with the topology $\tau_W$ generated by the seminorm $\|\cdot\|_W$. \\

We are ready to present the main result of this section.  Recall the definition of  topologically transitive operator, see Definition \ref{defn: trans op}.

\begin{proposition}\label{prop: prop 4.1}
 The operator $R_{V^*}$ is topologically transitive on $(L^2(\rho),\tau_W)$.
\end{proposition}
\begin{proof}
As observed earlier, we have that 
$$\lim_{n\to\infty}\left\|W\left(I-\sum_{j=-n}^n P_j\right)\right\|_\infty=0.$$
Hence, for each $X\in L^2(\rho)$ we get 
$$\begin{aligned}
\left\|X\left(I-\sum_{j=-n}^n P_j\right)\right\|_W^2&=\rho\left(X\left(I-\sum_{j=-n}^n P_j\right)W\left(I-\sum_{j=-n}^n P_j\right)X^*\right)\\
&\leq\|X\|^2_2\left\|\left(I-\sum_{j=-n}^n P_j\right)W\left(I-\sum_{j=-n}^n P_j\right)\right\|_\infty \\
&\leq \|X\|^2_2\left\|W\left(I-\sum_{j=-n}^n P_j\right)\right\|_\infty\to0,\quad n\to\infty.
\end{aligned}$$
Thus, given two non-empty $\tau_W-$open subsets $\mathcal{O}_1$ and $\mathcal{O}_2$ of $L^2(\rho)$, there exist some $F_1\in\mathcal{O}_1$, and $F_2\in\mathcal{O}_2$, and $k\in\NN$, such that $F_1\left(\sum\limits_{j={-k}}^k P_j\right)=F_1$, and $F_2\left(\sum\limits_{j={-k}}^k P_j\right)=F_2$. Hence, we obtain that
$$\begin{aligned}\|R^n_{V^*}(F_1)\|^2_W&=\rho(F_1(V^n)^*WV^nF_1^*)\\
&\leq\|V^nF_1^*\|_2^2\,\,\|W\|_\infty\\
&\leq\left\|V^n\left(\sum_{j=-k}^k P_j\right)F_1^*\right\|_2^2\|W\|_\infty\\
&\leq\left\|V^n\left(\sum_{j=-k}^k P_j\right)\right\|_\infty ^2\left\|F_1^*\right\|_2^2\|W\|_\infty\to0,\quad n\to\infty.
\end{aligned}
$$
Similarly, since 
$$\lim_{n\to\infty}\left\|V^{-n}\left(\sum_{j=-k}^k P_j\right)\right\|_\infty ^2=0,$$
we get
$$\lim_{n\to\infty}\|R^{-n}_{V^*}(F_2)\|_W^2=\lim_{n\to\infty}\|R_{(V^{-n})^*}(F_2)\|_W^2=0.$$
For each $n\in\NN$, set $X_n:=F_1+R^{-n}_{V^*}(F_2)$. Then, 
$$\|X_n-F_1\|_W=\|R^{-n}_{V^*}(F_2)\|_W\to0,\quad n\to\infty,$$
and
$$\|R^n_{V^*}(X_n)-F_2\|_W=\|R^n_{V^*}(F_1)\|_W\to0,\quad n\to\infty.$$
It follows that we can find some $n_0\in\NN$ such that $X_{n_0}\in\mathcal{O}_1$ and $R_{V^*}^{n_0}(X_{n_0})\in\mathcal{O}_2$. \end{proof}\bigskip
For each $n\in\NN$, we will now consider the cosine operator function generated by $R_{V^*}$, that is for every $n\in\NN$, we put $C^{(n)}=\frac{1}{2}\left(R_{V^*}^n+R_{V^*}^{-n}\right)$. The main idea for the proof of the next proposition is inspired by the proof of \cite[Theorem 5]{K}. 

\begin{proposition}
    The sequence $\{C^{(n)}\}_n$
is topologically transitive on $(L^2(\rho), \tau_W)$.\end{proposition}
\begin{proof}
    Let $\mathcal{O}_1$ and $\mathcal{O}_2$ be non-empty open subsets of $(L^2(\rho), \tau_W)$. As in the proof of Proposition \ref{prop: prop 4.1}, choose $F_1\in\mathcal{O}_1$ and $F_2\in \mathcal{O}_2$ such that $$F_1\left(\sum_{j=-k}^kP_j\right)=F_1 \quad\mbox{ and }\quad F_2\left(\sum_{j=-k}^kP_j\right)=F_2, \quad\mbox{ for some }k\in\NN.$$ By the same arguments as the ones in the proof of Proposition \ref{prop: prop 4.1}, we obtain that \begin{multline*}
        \lim_{n\to\infty}\|R^{n}_{V^*}(F_1)\|_W=\lim_{n\to\infty}\|R^{-n}_{V^*}(F_1)\|_W=\lim_{n\to\infty}\|R^{n}_{V^*}(F_2)\|_W\\=\lim_{n\to\infty}\|R^{-n}_{V^*}(F_2)\|_W=0.
    \end{multline*} For each $n\in\NN$, put $X_n:=F_1+R^{n}_{V^*}(F_2)+R^{-n}_{V^*}(F_2)$. Since $\|\cdot\|_W$ is a seminorm, then it satisfies the triangle inequality, hence it follows that $X_n\underset{\tau_W}{\to}F_1$ and $C^{(n)}(X_n)\underset{\tau_W}{\to}F_2$. Thus, there exists a $N\in\NN$ such that $C^{(N)}(\mathcal{O}_1)\cap\mathcal{O}_2\neq\emptyset$.
\end{proof}

\noindent\textbf{Data availability statement.} Data availability not applicable since no data sets were generated during this research.\\

\noindent\textbf{Declaration on conflict of interest.} The authors declare there is no conflict of interest in publishing the results obtained in this research. \\

\noindent\textbf{Funding.} SI and BDD are supported by the Ministry of Science, Technological Development and Innovations, Republic of Serbia, grant no. 451-03-66/2024-03/200029. BDD is supported by the bilateral project between Serbia and France (Generalized inverses on algebraic structures and applications), grant no. 337-00-93/2023-05/13. \\

\noindent{\bf{Acknowledgements.} } GB acknowledges that this work has been done within the activities of Gruppo UMI Teoria dell’Approssimazione e Applicazioni and of GNAMPA of the INdAM.

\end{document}